\documentclass[12pt]{article}
\usepackage{amsmath,amsthm,amssymb}
\usepackage[left=3cm,right=3cm, top=2.5cm,bottom=2.5cm,bindingoffset=0cm]{geometry}

\usepackage{amscd}

\usepackage{hyperref}

\usepackage{mathtools}

\usepackage[all,cmtip]{xy}

\usepackage{MnSymbol} 

\usepackage[table]{xcolor}

\usepackage{wrapfig}
\usepackage{graphicx}

\makeatletter
\def\th@plain{%
  \thm@notefont{}
  \itshape 
}
\def\th@definition{%
  \thm@notefont{}
  \normalfont 
}
\makeatother

\newtheorem{lemma}{Lemma}[section]
\newtheorem{proposition}[lemma]{Proposition}

\newtheorem{remark-definition}[lemma]{Remark-Definition}
\newtheorem{theorem}[lemma]{Theorem}
\newtheorem{corollary}[lemma]{Corollary}

\newtheorem{proposition-conjecture}[lemma]{Proposition-conjecture}

\theoremstyle{definition}

\newtheorem{definition}[lemma]{Definition}
\newtheorem{remark}[lemma]{Remark}

\newtheoremstyle{named}{}{}{\itshape}{}{\bfseries}{.}{.5em}{\thmnote{#3}}
\theoremstyle{named}
\newtheorem*{namedtheorem}{Theorem}

\usepackage{multirow}  

\usepackage[normalem]{ulem} 

\usepackage{tikz}
\usetikzlibrary{calc}

\definecolor{block}{RGB}{0,162,232}

\def\blockaux#1(#2,#3)#4(#5,#6){%
  \draw[fill={#1}]
  let \p1=(#2,#3),
      \p2=(#5,#6),
      \p3=(#2+#5,#3+#6),
      \p4=(#2+#5/2,#3+#6/2)
  in
    (\p1) rectangle (\p3)
    (\p4) node {$#4$}
  ;%
}

\begin{document}

\title{Shifts of semi-invariants and complete commutative \\
subalgebras in polynomial Poisson algebras }
\author{I.\,K.~Kozlov\thanks{No Affiliation, Moscow, Russia. E-mail: {\tt ikozlov90@gmail.com} }
}
\date{}

\maketitle

\begin{abstract} We study commutative subalgebras in the symmetric algebra $S(\mathfrak{g})$ of a finite-dimensional Lie algebra $\mathfrak{g}$. A.\,M.~Izosimov introduced extended Mischenko-Fomenko subalgebras $\tilde{\mathcal{F}}_a$ and gave a completeness criterion for them. We generalize his construction and extend Mischenko-Fomenko subalgebras with the shifts of all semi-invariants of $\mathfrak{g}$. We prove that the new commutative subalgebras have the same transcendence degree as $\tilde{\mathcal{F}}_a$. \end{abstract}

\tableofcontents

\section{Introduction}

In this paper we discuss a slight generalization of the results of A.\,M.~Izosimov from \cite{Izosimov14}. Hence, our introduction will be quite close to \cite{Izosimov14}.

Let $\mathfrak{g}$ be a finite-dimensional Lie algebra. For simplicity sake we assume that underlying field is $\mathbb{C}$, although all statements remain true for an arbitrary algebraically closed field $\mathbb{K}$ with $\operatorname{char} \mathbb{K} = 0$. The symmetric algebra $S(\mathfrak{g})$ can be naturally identified with the algebra of polynomials on the dual space $\mathfrak{g}^*$. The algebra $S(\mathfrak{g})$ has a natural Poisson bracket on it (called the \textbf{Lie-Poisson bracket}). The bracket is defined on linear functions by \[ \left\{\xi, \eta\right\} = [\xi, \eta], \qquad \forall \xi, \eta \in \mathfrak{g},\] and is extended to all polynomials by the Leibnitz identity.

We study the commutative subalgebras in $S(g)$. If $C \subset S(\mathfrak{g})$ is a commutative subalgebra, then \begin{equation}  \label{Eq:TrDeg}  \operatorname{tr. deg.} C \leq \frac{1}{2} \left( \dim \mathfrak{g} + \operatorname{ind} \mathfrak{g} \right).  \end{equation} If the equality holds  in \eqref{Eq:TrDeg}, then $C$ is called a \textbf{complete commutative subalgebra}. Each complete commutative subalgebra in $S(\mathfrak{g})$  can be regarded as an integrable system on the Poisson manifold $\mathfrak{g}^*$. There are several well-known methods for constructing commutative subalgebras in $ S(\mathfrak{g})$:

\begin{enumerate}

\item The \textbf{argument shift method} was introduced by A.\,S.~Mishchenko and A.\,T.~Fomenko in \cite{ArgShift}. It is a generalization of the S.\,V.~Manakov's construction \cite{Manakov76} for the Lie algebra $\operatorname{so}(n)$.  We describe a slight modification of the argument shift method, which is due to A.\,V.~Brailov (see A.\,V.~Bolsinov~\cite{Bolsinov14}).

Let $a \in \mathfrak{g}^*$ be an arbitrary regular element.  Denote the stabilizer of $x \in \mathfrak{g}^*$ w.r.t. the coadjoint representation  by \[ \mathfrak{g}_x  = \left\{ \xi \in \mathfrak{g} \, \, \bigr| \, \, \operatorname{ad}^*_{\xi}(x) = 0\right\}. \] There exist $m = \operatorname{ind} \mathfrak{g}$ local analytic invariants $f_1, \dots, f_m$ of the coadjoint representation defined in a neighbourhood of $a$ such that their differentials $df_i(a)$ form a basis of
$\mathfrak{g}_a$. Take the Taylor expansions of $f_i$ at $a$: \[ f_i(a + \lambda x) = \sum_{j=0}^{\infty} f_{ij} (x) \lambda^j,\]  where $f_{ij}(x)$ is a homogeneous polynomial in $x$ of degree $k$ and $\lambda$ is a formal parameter. 
\begin{definition} The algebra $\mathcal{F}_a \subset S(\mathfrak{g})$ generated by polynomials \[ f_{ij}, \qquad i = 1, \dots, \operatorname{ind} \mathfrak{g}, \, j > 0, \] is called the \textbf{algebra of (polynomial) shifts}. \end{definition}

It was proved in \cite{ArgShift} that $\mathcal{F}_a$ is a commutative subalgebra in $S(\mathfrak{g})$. Moreover, if $\mathfrak{g}$ is semisimple, then $\mathcal{F}_a$ is complete.

Completeness criterion for $\mathcal{F}_a$ was proved by A.\,V.~Bolsinov:
\begin{itemize}

\item Denote the set of singular elements\footnote{If the underlying field $\mathbb{K}$ is not algebraically closed, one should consider the singular set of $\mathfrak{g}^{\bar{\mathbb{K}}} = \mathfrak{g} \otimes_{\mathbb{K}} \bar{\mathbb{K}}$, where $\bar{\mathbb{K}}$ is the algebraic closure of $\mathbb{K}$. See \cite{BolsinovZuev}  for details.} in $\mathfrak{g}^*$ by \[ \operatorname{Sing} = \left\{ x \in \mathfrak{g}^*  \, \, \bigr| \, \, \dim \mathfrak{g}_x > \operatorname{ind} \mathfrak{g}\right\}.\]  In \cite{Bolsinov91} A.\,V.~Bolsinov proved that \textit{$\mathcal{F}_a$ is complete if and only if $\operatorname{codim} \operatorname{Sing} \geq 2$.}

\item In terms of the Jordan--Kronecker invariants of $\mathfrak{g}$ \textit{the algebra $\mathcal{F}_a$ is complete if and only if $\mathfrak{g}$ is of Kronecker type} (see  \cite{BolsZhang} for details).

\end{itemize}

\item In general, the algebra of shifts $\mathcal{F}_a$ is not complete. It was conjectured by A.\,S.~Mischenko and A.\,T.~Fomenko that \textit{for any finite-dimensional Lie algebra $\mathfrak{g}$, there exists a complete commutative subalgebra $C \subset S(\mathfrak{g})$}. This conjecture was proved by S.\,T.~Sadetov~\cite{Sadetov} (see also Bolsinov~\cite{Bolsinov05}). However, Sadetov's family $\mathcal{F} \subset S(\mathfrak{g})$ is essentially different from the algebra $\mathcal{F}_a$ of shifts. In particular, it is not commutative w.r.t. the following Poisson bracket. For any $a \in \mathfrak{g}^*$ there is a natural constant Poisson bracket  $\left\{ \cdot, \cdot\right\}_a$ on $S(\mathfrak{g})$ (the so-called \textbf{frozen argument bracket}). On linear functions it is given by \[ \left\{ \xi, \eta\right\}_a = \langle a, [\xi, \eta] \rangle, \qquad \xi, \eta \in \mathfrak{g},\] where $\langle \cdot, \cdot \rangle$ is the pairing between $\mathfrak{g}$ and $\mathfrak{g}^*$. Then the bracket $\left\{ \cdot, \cdot\right\}_a$ is extended to all polynomials by the Leibnitz identity. 

\begin{itemize}

\item The algebra of shifts $\mathcal{F}_a$ is commutative w.r.t. the both brackets $\left\{ \cdot, \cdot \right\}$ and $\left\{ \cdot, \cdot \right\}_a$. 

\item  The Sadetov's algebra $\mathcal{F}$ is commutative w.r.t. $\left\{ \cdot, \cdot \right\}$, but (in general) is not commutative w.r.t. $\left\{ \cdot, \cdot \right\}_a$. 

\end{itemize}

The following conjecture was stated by A.\,V.~Bolsinov and P.~Zhang in \cite{BolsZhang}.

\begin{namedtheorem}[Generalised argument shift conjecture] Let $\mathfrak{g}$ be an arbitrary finite-dimensional Lie algebra. Then for every regular element $a \in \mathfrak{g}^*$ there exists a complete family $\mathcal{G}_a \subset S(\mathfrak{g})$ of polynomials in bi-involution, i.e. in involution w.r.t. the two brackets $\left\{ \cdot, \cdot \right\}$ and $\left\{ \cdot, \cdot \right\}_a$. 
\end{namedtheorem}

\item In \cite{Izosimov14} A.\,M.~Izosimov introduced extended Mischenko-Fomenko subalgebras $\tilde{\mathcal{F}}_a$ and gave a completeness criterion for them. This subalgebras are defined as follows. Let $t = \dim \mathfrak{g} - \operatorname{ind} \mathfrak{g}$. Fix a basis in $\mathfrak{g}$, and let $c_{ij}^k$ be the structure constants in this basis. Let $\displaystyle \mathcal{A}_x = \left(\sum_k c_{ij}^k x_k \right)$ be the structure matrix of $\mathfrak{g}$. Consider the Pfaffians $\operatorname{Pf}C_{i_1 \dots i_t}(x)$ of the $t\times t$ minors of the structure matrix $\displaystyle \mathcal{A}_x$.  The \textbf{fundamental semi-invariant} $p_{\mathfrak{g}}(x)$ is the the greatest common divisor of the Pfaffians $\operatorname{Pf}C_{i_1 \dots i_t}(x)$. Next, consider the polynomial $p_{\mathfrak{g}}(a + \lambda x)$ and expand it in powers of $\lambda$:
\begin{equation} \label{Eq:ExpPg} p_{\mathfrak{g}} (a + \lambda x) = \sum_{i=1}^n p_i(x) \lambda^i, \end{equation}  where $n = \deg p_{\mathfrak{g}}$. 

\begin{definition} The \textbf{extended Mischenko-Fomenko subalgebra} $\tilde{\mathcal{F}}_a$ is a subalgebra in $S(\mathfrak{g})$ generated by all elements of the algebra of shifts $\mathcal{F}_a$ and the polynomials $p_1, \dots, p_n$. \end{definition}

Unlike the Sadetov's algebra $\mathcal{F}$, and similar to the algebra of shifts $\mathcal{F}_a$, the polynomials from $\tilde{\mathcal{F}}_a$ are in bi-involution.

\begin{theorem}[A.\,M.~Izosimov, \cite{Izosimov14}] For each regular $a \in \mathfrak{g}^*$, the extended Mischenko-Fomenko subalgebra $\tilde{\mathcal{F}}_a$ is commutative w.r.t. the both brackets $\left\{ \cdot, \cdot \right\}$ and $\left\{ \cdot, \cdot \right\}_a$.  \end{theorem} 

The completeness criterion is formulated is follows. 

\begin{itemize}

\item Let $\operatorname{Sing}_0$ be the union of all irreducible components of $\operatorname{Sing}$ that have dimension $\dim \mathfrak{g}-1$. (If  $\operatorname{codim} \operatorname{Sing} \geq 2$, then $\operatorname{Sing}_0 = \emptyset$.) 

\item Consider the subset \[ 
\operatorname{Sing}_{\mathfrak{b}} = \left\{y \in \operatorname{Sing}_0 \, \, \bigr| \, \, \mathfrak{g}_y \simeq \operatorname{aff}(1) \oplus \mathbb{C}^s  \right\} \subset \operatorname{Sing}_0, \] where $\operatorname{aff}(1)$ is the $2$-dimensional non-abelian Lie algebra. In can be shown that $\operatorname{Sing}_{\mathfrak{b}}$  is always Zariski open in $\operatorname{Sing}_0$, but might be empty.

\end{itemize}

\begin{theorem}[A.\,M.~Izosimov, \cite{Izosimov14}] \label{T:IzosComp} Let $\mathfrak{g}$ be a finite-dimensional complex Lie algebra and $a \in \mathfrak{g}^*$ be a regular element. The extended Mischenko-Fomenko subalgebra $\tilde{\mathcal{F}}_a$ is complete if and only if $\operatorname{Sing}_{\mathfrak{b}}$  is dense in $\operatorname{Sing}_0$. \end{theorem}

\begin{remark}  If $\operatorname{codim} \operatorname{Sing} \geq 2$, then   $\tilde{\mathcal{F}}_a$  is complete, because it contains the algebra of shifts $\mathcal{F}_a \subseteq \tilde{\mathcal{F}}_a$.  In the case $\operatorname{codim} \operatorname{Sing} \geq 2$ we formally have \begin{equation} \label{Eq:Empty} \operatorname{Sing}_0 = \operatorname{Sing}_{\mathfrak{b}} = \emptyset.\end{equation} We formally assume that \eqref{Eq:Empty} satisfies the condition ``$\operatorname{Sing}_{\mathfrak{b}}$  is dense in $\operatorname{Sing}_0$'', and Theorem~\ref{T:IzosComp} holds.  \end{remark}

\end{enumerate} 

A non-zero element $g\in S(\mathfrak{g})$ is a \textbf{semi-invariant with weight} $\chi_g \in \mathfrak{g}^*$ if \begin{equation} \label{Eq:SemiDef} \left\{ f, g\right\}_x = \chi_g(df) g(x)\end{equation} for any other $f \in S(\mathfrak{g})$. In this paper we study the following polynomial algebra. 

\begin{definition} The \textbf{algebra of shift of semi-invariants} $\mathcal{F}^{\mathrm{si}}_a$ is a subalgebra of $S(\mathfrak{g})$ generated by all elements of the algebra of shifts $\mathcal{F}_a$ and all the functions $g(x + \lambda a)$, for all semi-invariants $g(x)$ and all $\lambda \in \mathbb{C}$.  \end{definition}

The next statement was proved in \cite[Section 4]{Izosimov14}.

\begin{theorem}[A.\,M.~Izosimov, \cite{Izosimov14}] Let $\mathfrak{g}$ be a finite-dimensional Lie algebra, $a \in \mathfrak{g}^*$. Then for any two semi-invariants $f, g \in S(\mathfrak{g})$ and any $\lambda, \mu \in \mathbb{C}$, \[ \left\{ f(a + \lambda x), g(a + \lambda x)\right\} = 0, \qquad \left\{ f(a + \lambda x), g(a + \lambda x)\right\}_a = 0.\] \end{theorem}

Thus, similar to $\mathcal{F}_a$ and $\tilde{\mathcal{F}}_a$, the polynomials in  $\mathcal{F}^{\mathrm{si}}_a$ are in bi-involution.

\begin{corollary} For each regular $a \in \mathfrak{g}^*$ algebra of shift of semi-invariants $\mathcal{F}^{\mathrm{si}}_a$ is commutative w.r.t. the both brackets $\left\{ \cdot, \cdot \right\}$ and $\left\{ \cdot, \cdot \right\}_a$.  \end{corollary}

The difference between $\mathcal{F}^{\mathrm{si}}_a$  and $\tilde{\mathcal{F}}_a$ is quite simple:

\begin{itemize}

\item The extended Mischenko-Fomenko subalgebra $\tilde{\mathcal{F}}_a$ is generated by $\mathcal{F}_a$ and the shifts of the fundamental semi-invariant $p_{\mathfrak{g}}(a + \lambda x)$.

\item The algebra of shift of semi-invariants $\mathcal{F}^{\mathrm{si}}_a$ is generated by $\mathcal{F}_a$ and the shifts of all semi-invariants $g(a + \lambda x)$.

\end{itemize}

Obviously, $\tilde{\mathcal{F}}_a \subseteq \mathcal{F}^{\mathrm{si}}_a$. Our main result is

\begin{theorem} \label{T:MainTrDeg}  Let $\mathfrak{g}$ be a finite-dimensional Lie algebra and $a \in \mathfrak{g}^*$ be a regular element.  Then the algebra of shift of semi-invariants $\mathcal{F}^{\mathrm{si}}_a$ and the extended Mischenko-Fomenko subalgebra  $\tilde{\mathcal{F}}_a$ have the same transcendence degree: \begin{equation} \label{Eq:EqTrDeg}\operatorname{tr. deg.}  \mathcal{F}^{\mathrm{si}}_a = \operatorname{tr. deg.} \tilde{\mathcal{F}}_a.\end{equation}
\end{theorem}

Consider the following subspace of $T^*_x \mathfrak{g}^* \simeq \mathfrak{g}$ for the algebra of shifts $\mathcal{F}_a$:
\[d \mathcal{F}_a (x) = \left\{ df(x) \,\, \bigr| \,\, f \in \mathcal{F}_a\right\}, \qquad x \in \mathfrak{g}^*. \] Also consider similar subspaces $d \tilde{\mathcal{F}}_a (x)$ and $d\mathcal{F}^{\mathrm{si}}_a (x)$ for the algebras $\tilde{\mathcal{F}}_a$ and $\mathcal{F}^{\mathrm{si}}_a$ respectively. Since $\dim d \mathcal{F}_a (x) \leq \operatorname{tr. deg.} \mathcal{F}_a$ (and the similar inequality holds for  $ \tilde{\mathcal{F}}_a (x)$ and $\mathcal{F}^{\mathrm{si}}_a $ ) we get the following.

\begin{corollary} Let $a \in \mathfrak{g}^*$ be regular. Then \[d \tilde{\mathcal{F}}_a(x) \subseteq d\mathcal{F}^{\mathrm{si}}_a (x) \qquad \text {  and } \qquad \dim d\mathcal{F}^{\mathrm{si}}_a (x) \leq \operatorname{tr. deg.} \tilde{\mathcal{F}}_a \] for all $x \in \mathfrak{g}^*$. Thus, for all $x$ such that $\dim d \tilde{\mathcal{F}}_a(x)  = \operatorname{tr. deg.} \tilde{\mathcal{F}}_a$ we have \begin{equation} \label{Eq:EqualDF} d \tilde{\mathcal{F}}_a (x) = d\mathcal{F}^{\mathrm{si}}_a(x).\end{equation} In particular, \eqref{Eq:EqualDF} holds on an open dense subset of $\mathfrak{g}^*$. \end{corollary}

Roughly speaking, the distribution $d\mathcal{F}^{\mathrm{si}}_a $ is obtained from the distribution $d \tilde{\mathcal{F}}_a$ by ``increasing some singular subspaces''. The completeness criterion for $\mathcal{F}^{\mathrm{si}}_a $ and $\tilde{\mathcal{F}}_a$ is, obviously, the same.

\begin{corollary} \label{Cor:CompFsi} For a regular $a \in \mathfrak{g}^*$ the algebra of shift of semi-invariants $\mathcal{F}^{\mathrm{si}}_a$ is complete if and only if $\operatorname{Sing}_{\mathfrak{b}}$  is dense in $\operatorname{Sing}_0$. \end{corollary}

\par\medskip
 
\textbf{Acknowledgements.} The author would like to thank A.\,V.~Bolsinov and A.\,M.~Izosimov for useful comments.

\section{Roots of semi-invariants} \label{S:RootSemiInv}

In order to proof  Theorem~\ref{T:MainTrDeg} we need the following Lemma~\ref{L:EigenSemiInv} about the roots of semi-invariants. Below we consider Poisson brackets for (local) analytic functions. We extend the Lie--Poisson bracket $\left\{\cdot, \cdot\right\}$  and the  frozen argument  bracket $\left\{\cdot, \cdot\right\}_a$ to (local) analytic functions by \[  \left\{f,g \right\}  = \langle x, [ d f(x), dg(x)]\rangle, \qquad  \left\{f, g \right\}_a  = \langle a, [ df(x), dg(x)]\rangle.\]  Here we regard $df(x)$ and $dg(x)$ as elements of the Lie algebra $\mathfrak{g} \simeq T^*_x \mathfrak{g}^* $. The following statement was proved in \cite{Kozlov23JKRealization}, we repeat the proof for completeness sake.

\begin{lemma} \label{L:EigenSemiInv} Let $g(x)$ be a semi-invariant with weight $\chi_g$ on $\mathfrak{g}$ and $a \in \mathfrak{g}^*$ be an arbitrary element. Assume that $\lambda(x)$ is a solution of \begin{equation} \label{Eq:GEqLemma} g(x - \lambda(x) a) = 0,\end{equation} which is locally analytic on  $U \subset \mathfrak{g}^*$. Then  for any locally analytic function $h(x)$ on $U$ we have the equality 
\begin{equation} \label{Eq:EigenSemiInv} \left\{\lambda, h \right\}  = \lambda \left\{\lambda, h \right\}_a.\end{equation} 
\end{lemma}

\begin{remark} \label{Rem:Cond2Mat} If we consider two matrices \begin{equation} \label{Eq:Mat} \mathcal{A}_x = \left( \sum_i c^i_{jk} x_i \right), \qquad  \mathcal{A}_a = \left( \sum_i c^i_{jk} a_i \right),\end{equation} then \eqref{Eq:EigenSemiInv}  can be rewritten as \begin{equation} \label{Eq:DLCondWell} d\lambda (x) \in \operatorname{Ker} \left(\mathcal{A}_x - \lambda(x) \mathcal{A}_a\right).\end{equation}  The condition \eqref{Eq:DLCondWell} for the characterisitc numbers (=eigenvalues) of compatible Poisson brackets is well-known, for more details see  \cite{Kozlov23JKRealization}. 
\end{remark} 

\begin{proof}[Proof of Lemma~\ref{L:EigenSemiInv}] Let us use the explicit formulas for the Lie-Poisson and frozen argument brackets: \[  \left\{\lambda, h \right\}  = \langle x, [ d\lambda, dh]\rangle, \qquad  \left\{\lambda, h \right\}_a  = \langle a, [ d\lambda, dh]\rangle.\]   Then \begin{equation} \label{Eq:GPower} \left\{\lambda, h \right\}  - \lambda \left\{\lambda, h \right\}_a  = \langle x  - \lambda a, [ d\lambda, dh]\rangle. \end{equation} Differentianting \eqref{Eq:GEqLemma} by $x$ we get \[ d \lambda(x) =\frac{dg(x- \lambda a)}{\langle dg(x- \lambda a), a\rangle},\]  if $dg(x - \lambda a) \not = 0$. Since $dg(x - \lambda a) \not = 0$ for generic $x \in U$, \eqref{Eq:GPower} takes the form \begin{equation} \label{Eq:ProofDifL1}  \left\{\lambda, h \right\}  - \lambda \left\{\lambda, h\right\}_a = \frac{1}{\langle dg(x- \lambda a), a\rangle} \langle x  - \lambda a, [ dg(x-\lambda a), dh(x)]\rangle.\end{equation}  Note that the condition \eqref{Eq:SemiDef} in the definition of semi-invariants can be written as \begin{equation} \label{Eq:SemiInvCond} \langle x , [ df(x), dg(x)]\rangle = \chi_g(df(x)) \cdot  g(x) \end{equation} In local coordinates \eqref{Eq:SemiInvCond}  can be written as \begin{equation} \label{Eq:SumPartSemiEq1} \sum_{k, j}  c_{ij}^k x_k \frac{\partial g}{\partial x^j}(x) =  \left(\chi_g \right)_i \cdot g(x). \end{equation}  If in \eqref{Eq:SumPartSemiEq1} we replace $x$ with $x-\lambda a$, then for any function $h(x)$ that is analytic on $U$ we get \begin{equation} \label{Eq:SemiInvChange1}  \langle x - \lambda a, [ h(x), dg(x-\lambda a) ]\rangle = \chi_g(dh(x)) \cdot g(x - \lambda a).\end{equation} Substituting \eqref{Eq:SemiInvChange1} in \eqref{Eq:ProofDifL1} we get \[  \left\{\lambda, h \right\} - \lambda \left\{\lambda, h \right\}_a = -\frac{1}{\langle dg(x- \lambda a), a\rangle} \chi_g(dh(x)) \cdot g(x - \lambda a) = 0.\] Here the last equation holds by \eqref{Eq:GEqLemma}. Lemma~\ref{L:EigenSemiInv} is proved. \end{proof} 

\section{Proof of Theorem~\ref{T:MainTrDeg}}

Fix a regular $a \in \mathfrak{g}^*$. Let $g(x)$ be a semi-invariant on $\mathfrak{g}$.   Expand the polynomial $g(a + \lambda x)$ in powers of $\lambda$: \[ g(a + \lambda x) = \sum_{i=1}^m g_i (x) \lambda^i, \qquad m = \deg g(x).\] Obviously, the algebra $\mathcal{F}^{\mathrm{si}}_a$ contains the polynomials $g(a + \lambda x)$ for all $\lambda \in \mathbb{C}$ if and only if it contains all the elements $g_i(x), i=1,\dots, m$. Thus, it suffices to prove that each polynomial $g_i(x)$ is algebraic over the fraction field $\operatorname{Frac} \left(\tilde{\mathcal{F}}_a\right)$. If some $g_i(x)$ is not algebraic over $\operatorname{Frac} \left(\tilde{\mathcal{F}}_a\right)$, then \[ d g_i(x) \not \in d \tilde{\mathcal{F}}_a(x)\] on a Zariski open non-empty subset $U\subset \mathfrak{g}^*$. Therefore, it suffices to prove the following.

\begin{lemma} \label{L:GidFa} There exist a (non-empty) open subset $W_g \subset \mathfrak{g}^*$ such that \begin{equation} \label{Eq:DgFaCondLem} d g_i(x) \in d \tilde{\mathcal{F}}_a(x), \qquad i=1, \dots, m, \quad \forall x \in W_g.\end{equation} \end{lemma}

\begin{proof}[Proof of Lemma~\ref{L:GidFa} ] Consider the equation \begin{equation} \label{Eq:RootG} g(x - \lambda a) = 0.\end{equation}  First, we show that \eqref{Eq:DgFaCondLem} follows from a similar condition for the roots of \eqref{Eq:RootG}. We use the following trivial statement. 

\begin{proposition} \label{P:UgRoots}  Let $g(x) \in \mathbb{C}[x_1, \dots, x_n]$ be a homogeneous polynomial. Consider its factorization into irreducible factors: \[ g(x) = h_1^{k_1} \cdot \dots \cdot h_p^{k_p}. \] Then there exists a Zariski open non-empty set $U_g \subset \mathbb{C}^n$ such that the equation \eqref{Eq:RootG} has $d = \displaystyle \sum_{j=1}^p \deg h_j$ distinct roots $\lambda_1(x), \dots, \lambda_d(x)$ for all $x \in U_g$.  \end{proposition}

It is well-known that the roots $\lambda_i(x)$ of \eqref{Eq:RootG} are locally analytic functions on $U$, defined up to permutation (see e.g. \cite[p.362]{Hormanger}). Locally we have \[ g(x - \lambda a) = g(a) \prod_{i=1}^d (\lambda_i(x) - \lambda)^{s_i}.\] It is easy to see that the functions $g_1, \dots, g_m$ are, up to a constant factor, elementary symmetric polynomials of the functions $\lambda_1, \dots, \lambda_d$  taken with multiplicities. Thus, similar to \cite[Proposition 5.1]{Izosimov14} we get the following. 

\begin{proposition}\label{P:DiffLG}    Let $U_g$ be the subset from Proposition~\ref{P:UgRoots}, $\lambda_i(x),i=1,\dots, d$ be the roots of  \eqref{Eq:RootG}, and $g_j(x),j=1,\dots, m$ be given by \eqref{Eq:GPower}. Then \[ \operatorname{span} \left( d\lambda_1(x), \dots, d \lambda_d(x) \right) =  \operatorname{span} \left( dg_1(x), \dots, d g_m(x) \right), \qquad \forall x \in U_g. \] \end{proposition}

By Proposition~\ref{P:DiffLG}, in order to prove \eqref{Eq:DgFaCondLem}, it suffices to show that \begin{equation} \label{Eq:DLDFaProof} d \lambda_i(x) \in d \mathcal{F}_a(x), \qquad i=1, \dots, d, \end{equation}  on some open subset $W_g \subset U_g$. Next, let $U_{p_{\mathfrak{g}}}$ be the open subset for the fundamental semi-invariant $p_{\mathfrak{g}}$ from Proposition~\ref{P:UgRoots}. Consider an open subset $W_g \subset U_g \cap U_{p_{\mathfrak{g}}}$ such that the following $2$ conditions holds:

\begin{enumerate}

\item The line $x + \lambda a$ does not intersect the subset \[\operatorname{Sing}_1 = \operatorname{Sing} - \operatorname{Sing}_0\] (note that $\operatorname{codim} \operatorname{Sing}_1 \geq 2$). 

\item For each root $\lambda_i(x)$ of the equation \eqref{Eq:RootG} for the semi-invariant $g(x)$

\begin{itemize} 

\item either $p_{\mathfrak{g}}(x - \lambda_i(x) a) \equiv 0$ for all $x \in W_g$, 

\item or $p_{\mathfrak{g}}(x - \lambda_i(x) a) \not = 0$ for all $x \in W_g$.

\end{itemize}

\end{enumerate} 

We claim that \eqref{Eq:DLDFaProof} holds on $W_g$. There are $2$ cases:

\begin{itemize}

\item $p_{\mathfrak{g}}(x - \lambda_i(x) a) \equiv 0$ on $W_g$. Since $W_g \subset U_{p_{\mathfrak{g}}}$ we can apply Proposition~\ref{P:DiffLG}  for the fundamental semi-invariant $p_{\mathfrak{g}}(x)$. We get that  \[ d\lambda_i(x) \in  \operatorname{span} \left( dp_1(x), \dots, d p_n(x) \right) \subset d \tilde{\mathcal{F}}_a, \qquad \forall x \in W_g.\] Here the functions $p_j(x)$ are given by \eqref{Eq:ExpPg}.

\item $p_{\mathfrak{g}}(x - \lambda_i(x) a) \not = 0$ on $W_g$. By Lemma~\ref{L:EigenSemiInv} and Remark~\ref{Rem:Cond2Mat} \[ d\lambda_i (x)\in \operatorname{Ker} \left( \mathcal{A}_x - \lambda \mathcal{A}_a\right)\] By \cite[Remark 1]{BolsZhang} for any regular $a \in \mathfrak{g}^*$ the subspace \[ d \mathcal{F}_a(x) =  \sum_{x - \lambda a \not \in \operatorname{Sing} } \operatorname{Ker} \left( \mathcal{A}_x - \lambda \mathcal{A}_a\right),  \] for any $x \in \mathfrak{g}^*$ such that the line $x + \lambda a$ does not intersect $\operatorname{Sing}_1$. Thus, \[ d\lambda_i (x)  \in d \mathcal{F}_a(x)  \subseteq d \tilde{\mathcal{F}}_a(x), \qquad \forall x \in W_g.\]
\end{itemize}

We proved that \eqref{Eq:DLDFaProof} holds on $W_g$. Hence, by Proposition~\ref{P:DiffLG} the required condition \eqref{Eq:DgFaCondLem} also holds on $W_g$.  This proves Lemma~\ref{L:GidFa}, and thus Theorem~\ref{T:MainTrDeg}.  \end{proof}

\end{document}